\documentclass[numbers=enddot,12pt,final,onecolumn,notitlepage]{scrartcl}%
\usepackage[headsepline,footsepline,manualmark]{scrlayer-scrpage}
\usepackage{amssymb}
\usepackage{amsmath}
\usepackage{amsthm}
\usepackage{framed}
\usepackage{comment}
\usepackage{color}
\usepackage{hyperxmp}
\usepackage[breaklinks=True]{hyperref}
\usepackage[sc]{mathpazo}
\usepackage[T1]{fontenc}
\usepackage{needspace}
\usepackage{tabls}
\usepackage[type={CC}, modifier={zero}, version={1.0},]{doclicense}
\providecommand{\U}[1]{\protect\rule{.1in}{.1in}}
\theoremstyle{definition}
\newtheorem{theo}{Theorem}[section]
\newenvironment{theorem}[1][]
{\begin{theo}[#1]\begin{leftbar}}
{\end{leftbar}\end{theo}}
\newtheorem{lem}[theo]{Lemma}
\newenvironment{lemma}[1][]
{\begin{lem}[#1]\begin{leftbar}}
{\end{leftbar}\end{lem}}
\newtheorem{prop}[theo]{Proposition}
\newenvironment{proposition}[1][]
{\begin{prop}[#1]\begin{leftbar}}
{\end{leftbar}\end{prop}}
\newtheorem{defi}[theo]{Definition}

\newtheorem{remk}[theo]{Remark}

\newtheorem{coro}[theo]{Corollary}

\newtheorem{conv}[theo]{Convention}

\newtheorem{quest}[theo]{Question}
\newenvironment{question}[1][]
{\begin{quest}[#1]\begin{leftbar}}
{\end{leftbar}\end{quest}}
\newtheorem{warn}[theo]{Warning}
\newenvironment{warning}[1][]
{\begin{warn}[#1]\begin{leftbar}}
{\end{leftbar}\end{warn}}
\newtheorem{conj}[theo]{Conjecture}

\newtheorem{exam}[theo]{Example}

\let\sumnonlimits\sum
\let\prodnonlimits\prod
\let\cupnonlimits\bigcup
\let\capnonlimits\bigcap
\renewcommand{\sum}{\sumnonlimits\limits}
\renewcommand{\prod}{\prodnonlimits\limits}
\renewcommand{\bigcup}{\cupnonlimits\limits}
\renewcommand{\bigcap}{\capnonlimits\limits}
\newenvironment{noncompile}{}{}
\excludecomment{verlong}
\includecomment{vershort}
\excludecomment{noncompile}

\ihead{The entry sum of the inverse Cauchy matrix}
\ohead{page \thepage}
\begin{document}

\title{The entry sum of the inverse Cauchy matrix}
\author{Darij Grinberg\thanks{Drexel University, Korman Center, 15 S 33rd Street,
Philadelphia PA, 19104, USA}}
\date{\today}
\maketitle

\section{The Cauchy matrix}

Let $x_{1},x_{2},\ldots,x_{n}$ be $n$ numbers, and $y_{1},y_{2},\ldots,y_{n}$
be $n$ further numbers chosen such that all $n^{2}$ pairwise sums $x_{i}%
+y_{j}$ are nonzero\footnote{Algebraists can replace the words
\textquotedblleft number\textquotedblright\ and \textquotedblleft
nonzero\textquotedblright\ by \textquotedblleft element of a commutative
ring\textquotedblright\ and \textquotedblleft invertible\textquotedblright,
respectively. This generalization comes for free; we will not use anything
specific to any kind of numbers in our proofs.}. Consider the $n\times
n$-matrix%
\[
C:=\left(  \dfrac{1}{x_{i}+y_{j}}\right)  _{1\leq i\leq n,\ 1\leq j\leq
n}=\left(
\begin{array}
[c]{cccc}%
\dfrac{1}{x_{1}+y_{1}} & \dfrac{1}{x_{1}+y_{2}} & \cdots & \dfrac{1}%
{x_{1}+y_{n}}\\
\dfrac{1}{x_{2}+y_{1}} & \dfrac{1}{x_{2}+y_{2}} & \cdots & \dfrac{1}%
{x_{2}+y_{n}}\\
\vdots & \vdots & \ddots & \vdots\\
\dfrac{1}{x_{n}+y_{1}} & \dfrac{1}{x_{n}+y_{2}} & \cdots & \dfrac{1}%
{x_{n}+y_{n}}%
\end{array}
\right)  .
\]
This matrix $C$ is known as the \emph{Cauchy matrix}, and has been studied for
180 years\footnote{Many authors define it to have entries $\dfrac{1}%
{x_{i}-y_{j}}$ instead of $\dfrac{1}{x_{i}+y_{j}}$. This boils down to
replacing $y_{1},y_{2},\ldots,y_{n}$ by $-y_{1},-y_{2},\ldots,-y_{n}$.}. The
first significant result was the formula for its determinant:%
\begin{equation}
\det C=\frac{\prod_{1\leq i<j\leq n}\left(  \left(  x_{i}-x_{j}\right)
\left(  y_{i}-y_{j}\right)  \right)  }{\prod_{\left(  i,j\right)  \in\left\{
1,2,\ldots,n\right\}  ^{2}}\left(  x_{i}+y_{j}\right)  } \label{detC=}%
\end{equation}
found by Cauchy in 1841 \cite{Cauchy} (see, e.g., \cite[\S 1.3]{Prasolov} or
\cite[Exercise 6.18 or Exercise 6.64]{detnotes} for modern proofs). Newer
research focuses, e.g., on the LU decomposition \cite{GohKol}, positivity
properties \cite{Fiedler}, or generalizations \cite{IOTZ}. See \cite{Gow} for
more on the history of the topic and for its connections to Lagrange
interpolation (and for another proof of (\ref{detC=})). Applications range
from the theoretical (an equivalent version \cite[Lemma 5.15.3]{Etingof} of
(\ref{detC=}) is used in the classical representation theory of symmetric
groups) to the practical (computing the inverse $C^{-1}$ is a notoriously
ill-conditioned problem that is used as a canary for numerical instability
\cite{Todd}).

\section{The sum of the entries of the inverse}

The following curious result appears to be known since at least the 1940s:

\begin{theorem}
\label{thm.sum-inv}Assume that the matrix $C$ is invertible. Then, the sum of
all entries of its inverse $C^{-1}$ is $\sum_{k=1}^{n}x_{k}+\sum_{k=1}%
^{n}y_{k}$.
\end{theorem}

A natural, yet laborious approach to proving this theorem is to compute the
entries of $C^{-1}$ using (\ref{detC=}), and then to add them up. The
resulting sum can be seen (by a tricky induction) to simplify to $\sum
_{k=1}^{n}x_{k}+\sum_{k=1}^{n}y_{k}$. Some details of this proof can be found
in \cite[\S 1.2.3, Exercise 44]{Knuth-TAoCP1}. The proof given in
\cite[(13)]{Schechter} is simpler, avoiding the use of (\ref{detC=}) but
relying on Lagrange interpolation theory instead.

\begin{noncompile}
A natural approach to proving this theorem is to compute the entries of
$C^{-1}$; this can be done using (\ref{detC=}) and the standard formula
$A^{-1}=\dfrac{1}{\det A}\cdot\operatorname*{adj}A$ for the inverse of a
matrix $A$ (where $\operatorname*{adj}A$ denotes the adjugate of $A$). This
leads to the formula%
\[
\left(  C^{-1}\right)  _{i,j}=\dfrac{\prod_{k=1}^{n}\left(  x_{j}%
+x_{k}\right)  \left(  x_{k}+y_{i}\right)  }{\left(  x_{j}+y_{i}\right)
\cdot\prod_{k\in\left\{  1,2,\ldots,n\right\}  \setminus\left\{  j\right\}
}\left(  x_{j}-x_{k}\right)  \cdot\prod_{k\in\left\{  1,2,\ldots,n\right\}
\setminus\left\{  i\right\}  }\left(  y_{i}-y_{k}\right)  }%
\]
for the $\left(  i,j\right)  $-th entry of $C^{-1}$ (see \cite[\S 1.2.3,
Exercise 41]{Knuth-TAoCP1}). Summing over all $i,j$ leads to a complicated
expression which -- by a tricky induction -- can be seen to simplify to
$\sum_{k=1}^{n}x_{k}+\sum_{k=1}^{n}y_{k}$. Some details of this proof can be
found in \cite[\S 1.2.3, Exercise 44]{Knuth-TAoCP1}. The proof given in
\cite[(13)]{Schechter} is simpler, avoiding the use of (\ref{detC=}) but
relying on Lagrange interpolation theory instead.
\end{noncompile}

We propose a new proof of Theorem \ref{thm.sum-inv}, which reflects the
simplicity of the theorem. We let $A_{i,j}$ denote the $\left(  i,j\right)
$-th entry of any matrix $A$. The following simple lemma gets us half the way:

\begin{lemma}
\label{lem.AB}Let $A$ be an $n\times m$-matrix, and let $B$ be an $m\times
n$-matrix. Then,%
\[
\sum_{i=1}^{n}\ \ \sum_{j=1}^{m}\left(  x_{i}+y_{j}\right)  A_{i,j}%
B_{j,i}=\sum_{i=1}^{n}x_{i}\left(  AB\right)  _{i,i}+\sum_{j=1}^{m}%
y_{j}\left(  BA\right)  _{j,j}.
\]

\end{lemma}

\begin{proof}
[Proof of Lemma \ref{lem.AB}.]We have%
\begin{align*}
\sum_{i=1}^{n}\ \ \sum_{j=1}^{m}\left(  x_{i}+y_{j}\right)  A_{i,j}B_{j,i}  &
=\sum_{i=1}^{n}\ \ \sum_{j=1}^{m}x_{i}A_{i,j}B_{j,i}+\underbrace{\sum
_{i=1}^{n}\ \ \sum_{j=1}^{m}}_{=\sum_{j=1}^{m}\ \ \sum_{i=1}^{n}}%
y_{j}\underbrace{A_{i,j}B_{j,i}}_{=B_{j,i}A_{i,j}}\\
&  =\sum_{i=1}^{n}\ \ \sum_{j=1}^{m}x_{i}A_{i,j}B_{j,i}+\sum_{j=1}^{m}%
\ \ \sum_{i=1}^{n}y_{j}B_{j,i}A_{i,j}\\
&  =\sum_{i=1}^{n}x_{i}\underbrace{\sum_{j=1}^{m}A_{i,j}B_{j,i}}%
_{\substack{=\left(  AB\right)  _{i,i}\\\text{(by the definition
of}\\\text{the matrix product)}}}+\sum_{j=1}^{m}y_{j}\underbrace{\sum
_{i=1}^{n}B_{j,i}A_{i,j}}_{\substack{=\left(  BA\right)  _{j,j}\\\text{(by the
definition of}\\\text{the matrix product)}}}\\
&  =\sum_{i=1}^{n}x_{i}\left(  AB\right)  _{i,i}+\sum_{j=1}^{m}y_{j}\left(
BA\right)  _{j,j}.
\end{align*}

\end{proof}

\begin{proof}
[Proof of Theorem \ref{thm.sum-inv}.]Applying Lemma \ref{lem.AB} to $m=n$,
$A=C$ and $B=C^{-1}$, we obtain%
\begin{align*}
\sum_{i=1}^{n}\ \ \sum_{j=1}^{n}\left(  x_{i}+y_{j}\right)  C_{i,j}\left(
C^{-1}\right)  _{j,i}  &  =\sum_{i=1}^{n}x_{i}\underbrace{\left(
CC^{-1}\right)  _{i,i}}_{\substack{=1\\\text{(since }CC^{-1}\text{ is
the}\\\text{identity matrix)}}}+\sum_{j=1}^{n}y_{j}\underbrace{\left(
C^{-1}C\right)  _{j,j}}_{\substack{=1\\\text{(since }C^{-1}C\text{ is
the}\\\text{identity matrix)}}}\\
&  =\sum_{i=1}^{n}x_{i}+\sum_{j=1}^{n}y_{j}=\sum_{k=1}^{n}x_{k}+\sum_{k=1}%
^{n}y_{k}.
\end{align*}
However, the factor $\left(  x_{i}+y_{j}\right)  C_{i,j}$ on the left hand
side of this equality simplifies to $1$ (since the definition of $C$ yields
$C_{i,j}=\dfrac{1}{x_{i}+y_{j}}$). Thus, the left hand side of this equality
is $\sum_{i=1}^{n}\ \ \sum_{j=1}^{n}\underbrace{\left(  x_{i}+y_{j}\right)
C_{i,j}}_{=1}\left(  C^{-1}\right)  _{j,i}=\sum_{i=1}^{n}\ \ \sum_{j=1}%
^{n}\left(  C^{-1}\right)  _{j,i}$, which is clearly the sum of all entries of
$C^{-1}$. We have thus shown that the sum of all entries of $C^{-1}$ is
$\sum_{k=1}^{n}x_{k}+\sum_{k=1}^{n}y_{k}$. This proves Theorem
\ref{thm.sum-inv}.
\end{proof}

\section{Variants}

Theorem \ref{thm.sum-inv} was stated under the assumption that $C$ be
invertible. Using (\ref{detC=}), it is easy to see that this assumption is
equivalent to requiring that $x_{1},x_{2},\ldots,x_{n}$ be distinct and that
$y_{1},y_{2},\ldots,y_{n}$ be distinct\footnote{Algebraists working over an
arbitrary commutative ring should read \textquotedblleft
distinct\textquotedblright\ as \textquotedblleft strongly
distinct\textquotedblright\ (where two elements $a,b$ of a ring are said to be
\emph{strongly distinct} if their difference $a-b$ is invertible).}. It is not
hard to relieve Theorem \ref{thm.sum-inv} of this assumption: Just replace the
inverse $C^{-1}$ (which no longer exists) by the adjugate\footnote{The
\emph{adjugate} $\operatorname*{adj}A$ of an $n\times n$-matrix $A$ is the
$n\times n$-matrix whose $\left(  i,j\right)  $-th entry is $\left(
-1\right)  ^{i+j}\det\left(  A_{\sim j,\sim i}\right)  $, where $A_{\sim
j,\sim i}$ is the result of removing the $j$-th row and the $i$-th column from
$A$. Older texts often refer to the adjugate as the \textquotedblleft
classical adjoint\textquotedblright\ (or just as the \textquotedblleft
adjoint\textquotedblright, which however has another meaning as well).}
$\operatorname*{adj}C$ of the matrix $C$. The resulting theorem is as follows:

\begin{theorem}
\label{thm.sum-adj}The sum of all entries of the adjugate matrix
$\operatorname*{adj}C$ is $\left(  \sum_{k=1}^{n}x_{k}+\sum_{k=1}^{n}%
y_{k}\right)  \det C$.
\end{theorem}

\begin{proof}
Similar to our above proof of Theorem \ref{thm.sum-inv}.\footnote{I wrote up
this proof in much more detail in \cite[solution to Exercise 6.69
\textbf{(a)}]{detnotes}.} Use the classical result that $C\cdot
\operatorname*{adj}C=\operatorname*{adj}C\cdot C=\det C\cdot I_{n}$ (where
$I_{n}$ denotes the $n\times n$ identity matrix).
\end{proof}

Theorem \ref{thm.sum-adj} can be transformed even further:

\begin{theorem}
\label{thm.border-det}Let $D$ be the $\left(  n+1\right)  \times\left(
n+1\right)  $-matrix obtained from $C$ by inserting a row full of $1$'s at the
very bottom and a column full of $1$'s at the very right, and putting $0$ in
the bottom-right corner:%
\[
D=\left(
\begin{array}
[c]{ccccc}%
\dfrac{1}{x_{1}+y_{1}} & \dfrac{1}{x_{1}+y_{2}} & \cdots & \dfrac{1}%
{x_{1}+y_{n}} & 1\\
\dfrac{1}{x_{2}+y_{1}} & \dfrac{1}{x_{2}+y_{2}} & \cdots & \dfrac{1}%
{x_{2}+y_{n}} & 1\\
\vdots & \vdots & \ddots & \vdots & \vdots\\
\dfrac{1}{x_{n}+y_{1}} & \dfrac{1}{x_{n}+y_{2}} & \cdots & \dfrac{1}%
{x_{n}+y_{n}} & 1\\
1 & 1 & \cdots & 1 & 0
\end{array}
\right)  .
\]
Then,%
\[
\det D=-\left(  \sum_{k=1}^{n}x_{k}+\sum_{k=1}^{n}y_{k}\right)  \cdot\det C.
\]

\end{theorem}

\begin{proof}
[Proof sketch.]This follows from Theorem \ref{thm.sum-adj} using the following
more general fact: If $A$ is any $n\times n$-matrix, and if $B$ is the
$\left(  n+1\right)  \times\left(  n+1\right)  $-matrix obtained from $A$ in
the same way as $D$ was obtained from $C$ (that is, by inserting a row full of
$1$'s at the very bottom and a column full of $1$'s at the very right, and
putting $0$ in the bottom-right corner), then%
\[
\det B=-s,
\]
where $s$ is the sum of all entries of $\operatorname*{adj}A$. This fact, in
turn, can be proved by Laplace expansion of $\det B$ along the last row
(followed by expanding each cofactor along the last column). We refer to
\cite[solution to Exercise 6.69 \textbf{(c)}]{detnotes} for all details.
\end{proof}

Theorem \ref{thm.border-det} appears in \cite[Chapter XI, Exercise
43]{MuiMet60}; we know nothing more about its origins.

\section{Two little exercises}

For all its aid in our proof, it appears that Lemma \ref{lem.AB} is a
one-trick pony: We are unaware of any other interesting results whose proofs
it simplifies. The sum of all entries of a matrix is not generally a
particularly well-behaved quantity (unlike the sum of its \textbf{diagonal}
entries, which is known as the trace and has many good properties). However,
some experimentation has led us to a surprising (if not very deep) twin to
Theorem \ref{thm.sum-inv}.

We assume that $x_{1},x_{2},\ldots,x_{n}$ and $y_{1},y_{2},\ldots,y_{n}$ are
\textbf{real} numbers (and that $n\geq1$). Consider the $n\times n$-matrix%
\[
F:=\left(  \min\left\{  x_{i},y_{j}\right\}  \right)  _{1\leq i\leq n,\ 1\leq
j\leq n}=\left(
\begin{array}
[c]{cccc}%
\min\left\{  x_{1},y_{1}\right\}  & \min\left\{  x_{1},y_{2}\right\}  & \cdots
& \min\left\{  x_{1},y_{n}\right\} \\
\min\left\{  x_{2},y_{1}\right\}  & \min\left\{  x_{2},y_{2}\right\}  & \cdots
& \min\left\{  x_{2},y_{n}\right\} \\
\vdots & \vdots & \ddots & \vdots\\
\min\left\{  x_{n},y_{1}\right\}  & \min\left\{  x_{n},y_{2}\right\}  & \cdots
& \min\left\{  x_{n},y_{n}\right\}
\end{array}
\right)  .
\]
Thus, $F$ is obtained from $C$ by replacing the \textquotedblleft inverted
sums\textquotedblright\ $\dfrac{1}{x_{i}+y_{j}}$ by the minima $\min\left\{
x_{i},y_{j}\right\}  $\ \ \ \ \footnote{This can be seen as an instance of
tropicalization (see, e.g., \cite{NoumiYamada}). More precisely,
tropicalization (the sort that replaces $+$ and $\cdot$ by $\max$ and $+$)
would replace $\dfrac{1}{x_{i}+y_{j}}$ by $-\max\left\{  x_{i},y_{j}\right\}
$; but this turns into $\min\left\{  x_{i},y_{j}\right\}  $ if we multiply all
our numbers $x_{1},x_{2},\ldots,x_{n},y_{1},y_{2},\ldots,y_{n}$ by $-1$.}. It
would almost be too much to ask for $F^{-1}$ to have properties comparable to
those of $C^{-1}$. But in fact, it behaves even better:

\begin{proposition}
Assume that $F$ is invertible. Then:

\begin{enumerate}
\item[\textbf{(a)}] The sum of all entries of $F^{-1}$ is $\dfrac{1}%
{\min\left\{  x_{1},x_{2},\ldots,x_{n},y_{1},y_{2},\ldots,y_{n}\right\}  }$.

\item[\textbf{(b)}] Assume that $x_{1}\leq x_{2}\leq\cdots\leq x_{n}$ and
$y_{1}\leq y_{2}\leq\cdots\leq y_{n}$ and $x_{1}\leq y_{1}$. Then, for each
$j\in\left\{  1,2,\ldots,n\right\}  $, the sum of all entries in the $j$-th
column of $F^{-1}$ is $\dfrac{1}{x_{1}}$ if $j=1$, and is $0$ if $j>1$.
\end{enumerate}
\end{proposition}

\begin{noncompile}
Proof idea:

\textbf{(b)} For each $j\in\left\{  1,2,\ldots,n\right\}  $, we have%
\[
\delta_{1,j}=\left(  FF^{-1}\right)  _{1,j}=\sum_{k=1}^{n}\underbrace{F_{1,k}%
}_{\substack{=\min\left\{  x_{1},y_{k}\right\}  \\=x_{1}}}\left(
F^{-1}\right)  _{k,j}=x_{1}\sum_{k=1}^{n}\left(  F^{-1}\right)  _{k,j}.
\]
By plugging $j=1$, we conclude that $x_{1}\neq0$. Thus, we can divide this
equality by $x_{1}$ and find%
\[
\sum_{k=1}^{n}\left(  F^{-1}\right)  _{k,j}=\dfrac{\delta_{1,j}}{x_{j}}.
\]
But this is precisely the claim of \textbf{(b)}.

\textbf{(a)} We can WLOG assume that $x_{1}\leq x_{2}\leq\cdots\leq x_{n}$
(else, permute the $x$'s) and $y_{1}\leq y_{2}\leq\cdots\leq y_{n}$ (else,
permute the $y$'s) and $x_{1}\leq y_{1}$ (else, interchange the $x$'s with the
$y$'s). Now apply part \textbf{(b)}.
\end{noncompile}

The proof of this proposition is another neat exercise in working with inverse
matrices -- one we do not want to spoil for the reader. As with $C$, computing
the determinant is not necessary. However, it is computable, and the result is
another nice exercise:

\begin{proposition}
Assume that $x_{1}\leq x_{2}\leq\cdots\leq x_{n}$ and $y_{1}\leq y_{2}%
\leq\cdots\leq y_{n}$. For any $i,j\in\left\{  1,2,\ldots,n\right\}  $, set
$f_{i,j}:=\min\left\{  x_{i},y_{j}\right\}  $. Then,%
\begin{equation}
\det F=f_{1,1}\cdot\prod_{k=2}^{n}\left(  f_{k,k}-f_{k,k-1}-f_{k-1,k}%
+f_{k-1,k+1}\right)  . \label{eq.detF=}%
\end{equation}

\end{proposition}

\begin{noncompile}
Proof idea: Subtract each row of $F$ from the next row and each column from
the next column (and do this all simultaneously). Let $F^{\prime}$ be the
resulting matrix. Then, $\det\left(  F^{\prime}\right)  =\det F$. Now we want
to show that $\det\left(  F^{\prime}\right)  $ equals the product of the
diagonal entries of $F^{\prime}$.

To do so, we observe that the entries $f_{i,j}^{\prime}$ of $F^{\prime}$ satisfy:

\begin{itemize}
\item If $x_{k}\leq y_{k}$, then $f_{i,j}^{\prime}=0$ for all $i\leq k$ and
$j>k$.

\item If $y_{k}\leq x_{k}$, then $f_{i,j}^{\prime}=0$ for all $i>k$ and $j\leq
k$.
\end{itemize}

Hence, $\det\left(  F^{\prime}\right)  $ can be computed by iteratively
applying the \textquotedblleft determinant of block-triangular
matrix\textquotedblright\ formula, and at the end you get the product of all
diagonal entries.

There might be a more elegant way to see this.
\end{noncompile}

Note that the product on the right hand side of (\ref{eq.detF=}) will often be
$0$ if the $x_{i}$ and the $y_{j}$'s are ordered in an \textquotedblleft
insufficiently balanced\textquotedblright\ way (e.g., if there are more than
two $y_{j}$'s between two consecutive $x_{i}$'s). We leave it to the reader to
establish more precise criteria for $\det F$ to be $0$.


\begin{thebibliography}{99}                                                                                               %


\bibitem {Cauchy}%
\href{https://gallica.bnf.fr/ark:/12148/bpt6k90204r/f176}{Augustin-Louis
Cauchy, \textit{Memoire sur les fonctions altern\'{e}es et sur les sommes
altern\'{e}es}, Exercices d'analyse et de phys. math., ii (1841). pp.
151--159. Reprinted in: \OE uvres completes, 2e ser., tome xii.}

\bibitem {detnotes}\href{https://arxiv.org/abs/2008.09862v3}{Darij Grinberg,
\textit{Notes on the combinatorial fundamentals of algebra}, 15 September
2022, arXiv:2008.09862v3}.

\bibitem {Etingof}\href{https://math.mit.edu/~etingof/repb.pdf}{Pavel Etingof,
Oleg Golberg, Sebastian Hensel, Tiankai Liu, Alex Schwendner, Dmitry Vaintrob,
Elena Yudovina, \textit{Introduction to Representation Theory}, with
historical interludes by Slava Gerovitch, Student Mathematical Library
\textbf{59}, AMS 2011.}

\bibitem {Fiedler}\href{https://doi.org/10.1016/j.laa.2009.08.014}{Miroslav
Fiedler, \textit{Notes on Hilbert and Cauchy matrices}, Linear Algebra and its
Applications \textbf{432} (2010), pp. 351--356.}

\bibitem {GohKol}\href{https://doi.org/10.1007/BF01229504}{Israel Gohberg,
Israel Koltracht, \textit{Triangular factors of Cauchy and Vandermonde
matrices}, Integral Eq. Operator Theory \textbf{26} (1996), pp. 46--59.}

\bibitem {Gow}%
\href{https://www.maths.tcd.ie/pub/ims/bull28/bull28_45-52.pdf}{R. Gow,
\textit{Cauchy's matrix, the Vandermonde matrix and polynomial interpolation},
IMS Bulletin \textbf{28} (1992), pp. 45--52.}

\bibitem {IOTZ}\href{https://doi.org/10.1016/j.aam.2005.07.001}{Masao
Ishikawa, Soichi Okada, Hiroyuki Tagawa, Jiang Zeng, \textit{Generalizations
of Cauchy's determinant and Schur's Pfaffian}, Advances in Applied Mathematics
\textbf{36} (2006), pp. 251--287}.

\bibitem {Knuth-TAoCP1}Donald Ervin Knuth, \textit{The Art of Computer
Programming, volume 1: Fundamental Algorithms}, 3rd edition, Addison--Wesley
1997.\newline See \url{https://www-cs-faculty.stanford.edu/~knuth/taocp.html}
for errata.

\bibitem {MuiMet60}Thomas Muir, \textit{A Treatise on the Theory of
Determinants}, revised and enlarged by William H. Metzler, Dover 1960.

\bibitem {NoumiYamada}\href{http://arxiv.org/abs/math-ph/0203030v2}{Masatoshi
Noumi, Yasuhiko Yamada, \textit{Tropical Robinson-Schensted-Knuth
correspondence and birational Weyl group actions}, arXiv:math-ph/0203030v2.}

\bibitem {Prasolov}\href{https://bookstore.ams.org/mmono-134}{Viktor V.
Prasolov, \textit{Problems and Theorems in Linear Algebra}, Translations of
Mathematical Monographs, vol. \#134, AMS 1994}.\newline See
\url{https://staff.math.su.se/mleites/books/prasolov-1994-problems.pdf} for a
preprint.\newline See also
\url{https://drive.google.com/open?id=0B2UfTLwpN9okblBJbGxOZXc4Rm8} for a
newer edition in Russian.

\bibitem {Schechter}%
\href{https://doi.org/10.1090/S0025-5718-1959-0105798-2}{Samuel Schechter,
\textit{On the Inversion of Certain Matrices}, Math. Comp. \textbf{13} (1959),
pp. 73--77.}

\bibitem {Todd}%
\href{https://nvlpubs.nist.gov/nistpubs/jres/65B/jresv65Bn1p19_A1b.pdf}{John
Todd, \textit{Computational Problems Concerning the Hilbert Matrix}, Journal
of Research of the National Bureau of Standards--B \textbf{65} (1961), no. 1,
pp. 19--22.}
\end{thebibliography}
\end{document}